\documentclass[11pt,twoside]{amsart}
\usepackage[latin1]{inputenc}
\usepackage[OT1]{fontenc}
\usepackage{amsmath}
\usepackage{amsthm}
\usepackage{amssymb}
\usepackage[all]{xy}

\newtheorem{thm}{Theorem}[section]

\newtheorem{prop}[thm]{Proposition}

\newtheorem{lem}[thm]{Lemma}
\newtheorem{cor}[thm]{Corollary}

\numberwithin{equation}{section}

\theoremstyle{definition}
\newtheorem{definition}[thm]{Definition}
\newtheorem{remark}[thm]{Remark}
\newtheorem{ex}[thm]{Example}

\newcommand{\qqed}{\hspace*{\fill}$\Box$}

\newcommand{\im}{\operatorname{im}}
\newcommand{\Db}{{\rm D}^{\rm b}}

\newcommand{\Aut}{{\rm Aut}}

\newcommand{\Hom}{{\rm Hom}}

\newcommand{\Perf}{{\rm Perf}}

\newcommand{\Inf}{{\bf Inf}}
\newcommand{\Forg}{{\bf Forg}}

\newcommand{\id}{{\rm id}}

\newcommand{\coker}{{\rm coker}}
\newcommand{\coim}{{\rm coim}}

\newcommand{\ka}{{\mathcal A}}
\newcommand{\kb}{{\mathcal B}}
\newcommand{\kc}{{\mathcal C}}

\newcommand{\kf}{{\mathcal F}}

\newcommand{\kl}{{\mathcal L}}

\newcommand{\ko}{{\mathcal O}}
\newcommand{\kp}{{\mathcal P}}
\newcommand{\kq}{{\mathcal Q}}
\newcommand{\kt}{{\mathcal T}}

\newcommand{\IC}{\mathbb{C}}

\newcommand{\IZ}{\mathbb{Z}}

\DeclareMathOperator{\Coh}{\bf{Coh}}
\DeclareMathOperator{\QCoh}{\bf{QCoh}}
\DeclareMathOperator{\(Q)Coh}{\bf{(Q)Coh}}

\renewcommand{\to}{\xymatrix@1@=15pt{\ar[r]&}}
\renewcommand{\rightarrow}{\xymatrix@1@=15pt{\ar[r]&}}
\renewcommand{\mapsto}{\xymatrix@1@=15pt{\ar@{|->}[r]&}}
\renewcommand{\twoheadrightarrow}{\xymatrix@1@=15pt{\ar@{->>}[r]&}}
\renewcommand{\hookrightarrow}{\xymatrix@1@=15pt{\ar@{^(->}[r]&}}
\newcommand{\congpf}{\xymatrix@1@=15pt{\ar[r]^-\sim&}}
\renewcommand{\cong}{\simeq}

\begin{document}

\title[Linearisations of triangulated categories...]{Linearisations of triangulated categories with respect to finite group actions}
\author[P.\ Sosna]{Pawel Sosna}

\address{Dipartimento di Matematica ``F.\ Enriques'', Universit\`a degli Studi di Milano, Via Cesare Saldini 50,
20133 Milano, Italy}
\email{pawel.sosna@guest.unimi.it}



\begin{abstract} \noindent
Given an action of a finite group on a triangulated category, we investigate under which conditions one can construct a linearised triangulated category using DG-enhancements. In particular, if the group is a finite group of automorphisms of a smooth projective variety and the category is the bounded derived category of coherent sheaves, then our construction produces the bounded derived category of coherent sheaves on the smooth quotient variety resp.\ stack. We also consider the action given by the tensor product with a torsion canonical bundle and the action of a finite group on the category generated by a spherical object.

\vspace{-2mm}\end{abstract}
\maketitle

\maketitle
\let\thefootnote\relax\footnotetext{This work was supported by the research grant SO 1095/1-1 of the DFG (German Research Foundation).}

\section{Introduction}
Triangulated categories are ubiquitous in several areas of mathematics. However, it is also well-known that these categories are not rigid enough to perform certain operations: For example, the category of exact functors between triangulated categories is not itself triangulated. This note is concerned with defining a linearised triangulated category $\kt^G$ when the action of a group on a triangulated category $\kt$ is given.

One possible motivation for this comes from geometry. The bounded derived category $\Db(X)$ of a smooth projective variety $X$ has drawn a lot of attention in recent years since it encodes a lot of interesting geometric information which is not visible when working with the abelian category of sheaves. 
%For example, Fourier--Mukai partners of $X$, that is, varieties $Y$ such that $\Db(X)\cong \Db(Y)$, are geometrically different objects but still have a close connection with $X$. 
The group of autoequivalences $\Aut(\Db(X))$ is also a very interesting object to study. It is completely understood for varieties with ample canonical (or anticanonical) bundle, but its structure is unknown in general. Of course, similar questions can be asked about any triangulated category.

If the triangulated category happens to be $\Db(X)$ and the group $G$ is contained in $\Aut(X)$ and is finite, then a reasonable construction should produce $\Db([X/G])$, where $[X/G]$ is the smooth quotient variety resp.\ stack. On the other hand, if $G$ is not some group of automorphisms one might hope to extract some interesting geometry out of its action on $\Db(X)$. Unfortunately, we do not have many examples of finite groups acting on $\Db(X)$ at the moment, but since there are various sources for triangulated categories it is reasonable to hope that the techniques developed in this note will also be applicable in other situations.

The basic idea of the construction is to use linearisations. This is well-established for sheaves (or modules): If $G \subset \Aut(X)$ is finite, then the abelian category $\Coh^G(X)$ of linearised sheaves on $X$ is equivalent to $\Coh([X/G])$. We would like to have something similar for triangulated categories. However, there are two basic problems. The first is that a reasonable notion of a group acting on a category has to assign an autoequivalence to any group element (and the assignment is subject to some conditions), but the group $\Aut(\kt)$ is the set of autoequivalences modulo isomorphisms. Furthermore, it is fairly easy to see that the category of linearised objects (with respect to a reasonable group action) of $\kt$ is not necessarily triangulated, because cones are not functorial (a problem alluded to in the first paragraph). 

The remedy is to consider triangulated categories which are homotopy categories of pretriangulated DG-categories and autoequivalences which come from equivalences on the DG-level. Our main result can be roughly summarised as follows, see Proposition \ref{classical-picture} and Corollary \ref{Enriques-K3}.

\begin{thm}
Let $G$ be a finite group acting on a triangulated category $\kt\cong H^0(\ka)$ which is the homotopy category of a pretriangulated DG-category $\ka$ and such that the action of $G$ comes from an action on $\ka$. Then a linearised triangulated category $\kt_\ka^G$ can be constructed. If $\kt\cong\Db(X)$ and $G$ is a finite group of automorphisms of $X$, then $\Db(X)^G\cong \Db([X/G])$. Given a variety $S$ with a canonical bundle $\omega_S$ which is torsion of order $n$, the triangulated category linearised with respect to the action of $\IZ/n\IZ$, where one identifies $1$ with the autoequivalence given by the tensor product with $\omega_S$, is equivalent to the bounded derived category of the canonical cover.
\end{thm}

It is in general not clear whether the above result depends on the choice of the category $\ka$ (so the above statements have to be read as involving specific choices of $\ka$), but see Proposition \ref{ind-enh} for a partial result. In any case, choosing a fairly natural $\ka$ we can prove the following.
For example, given a Fourier--Mukai partner $Y$ of $X$ and a group $G$ acting on $Y$, the linearised category $\Db(X)^G$ with respect to the action of $G$ induced by the Fourier--Mukai equivalence is equivalent to $\Db([Y/G])$, see Corollary \ref{FM-quotient}.

Lastly, the category generated by a spherical object admits actions of finite groups and we prove that the spherical object becomes an exceptional one in the linearised category, see Proposition \ref{sph-exc}.

\begin{remark}
If the group was not finite, one would have to adjust certain things: For example, we would have to work with (the derived category of) quasi-coherent sheaves, but the approach would still work.
\end{remark}

The note is organised as follows. In Section 2 we recall some basic facts about DG-categories, define the linearised triangulated category in the following section, consider the above mentioned geometric situation in Section 4 and look at new examples in the last section.

\smallskip
\noindent{\bf{Conventions.}} From Section 3 on we work over the field of complex numbers (although most of the results hold over an arbitrary field provided the order of the group is prime to the characteristic). All functors between derived categories are assumed to be exact. Unless stated otherwise all considered groups are finite. 
\smallskip

\noindent{\bf{Acknowledgements.}} I thank David Ploog and Paolo Stellari for useful discussions and for comments on a preliminary version of this paper and the department of mathematics and the complex geometry group of the Universit\`{a} degli Studi di Milano for their hospitality.

\section{Differential graded categories}
In this section we recall the necessary notions and facts from the theory of differential graded categories. For details see e.g.\ \cite{Drinfeld}, \cite{Keller} or \cite{LO}.

\begin{definition}
A \emph{differential graded category} or \emph{DG-category} over a field $K$ is a $K$-linear additive category $\ka$ such that for any two objects $X, Y \in \ka$ the space of morphisms $\Hom(X,Y)$ is a complex, the composition of morphisms 
\[\Hom(X,Y)\otimes \Hom(Y,Z) \rightarrow \Hom(X,Z)\]
is a chain map and the identity with respect to the composition is closed of degree $0$.
\end{definition}

\begin{ex}
The most basic example of a $K$-linear DG-category is the category of complexes of $K$-vector spaces. For two complexes $X$ and $Y$ we define $\Hom(X,Y)^n$ to be the $K$-vector space formed by families $\alpha=(\alpha^p)$ of morphisms $\alpha^p: X^p \rightarrow Y^{p+n}$, $p \in \IZ$. We define $\Hom_{DG}(X,Y)$ to be the graded $K$-vector space with components $\Hom(X,Y)^n$ and whose differential is given by
\[d(\alpha)=d_Y\circ \alpha -(-1)^n\alpha \circ d_X.\]
The DG-category $C_{DG}(K)$ has as objects complexes and the morphisms are defined by
\[C_{DG}(K)(X,Y)=\Hom_{DG}(X,Y).\]
Of course, starting with the category of complexes over an arbitrary $K$-linear abelian (or additive) category one can associate a DG-category to it in a similar manner.

Clearly, we get back the usual category of complexes by taking as morphisms only the closed morphisms of degree zero and we get the usual homotopy category if we replace $\Hom_{DG}(X,Y)$ by $\ker(d^0)/\im(d^{-1})$. 
\end{ex}

A \emph{DG-functor} $\Phi\colon \ka \rightarrow \kb$ between DG-categories $\ka$ and $\kb$ is by definition required to be compatible with the structure of complexes on the spaces of morphisms. If $\Phi,\Psi\colon \ka \rightarrow \kb$ are two DG-functors, then we define \emph{the complex of graded morphisms} $\Hom(\Phi, \Psi)$ to be the complex whose $n$th component is the space formed by families of morphisms $\phi_X \in \Hom_\kb(\Phi(X),\Psi(X))^n$ such that $(\Psi\alpha)(\phi_X)=(\phi_Y)(\Phi\alpha)$ for all $\alpha \in \Hom_\ka(X,Y)$, where $X,Y \in \ka$. The differential is given by that of $\Hom_\kb(\Phi(X),\Psi(X))$. Using this we define the DG-category of DG-functors from $\ka$ to $\kb$, denoted by $\Hom(\ka,\kb)$, to be the category with DG-functors as objects and the above described spaces as morphisms. Note that the DG-functors between $\ka$ and $\kb$ are precisely the closed morphisms of degree zero in $\Hom(\ka,\kb)$.

To any DG-category $\ka$ one can naturally associate two other categories: Firstly, there is the \emph{graded category} $Ho^\bullet(\ka)=H^\bullet(\ka)$ having the same objects as $\ka$ and where the space of morphisms between two objects $X, Y$ is by definition the direct sum of the cohomologies of the complex $\Hom_\ka(X, Y)$. Secondly, restricting to the cohomology in degree zero we get the \emph{homotopy category} $Ho(\ka)=H^0(\ka)$. 

\begin{definition}\label{quasi}
A DG-functor $\Phi\colon \ka \rightarrow \kb$ is \emph{quasi fully faithful} if for any two objects $X,Y$ in $\ka$ the map 
\[\Hom(X,Y) \rightarrow \Hom(\Phi(X), \Phi(Y))\]
is a quasi-isomorphism and $\Phi$ is a \emph{quasi-equivalence} if in addition the induced functor $H^0(\Phi)$ is essentially surjective. Two DG-categories $\ka$ and $\kb$ are called \emph{quasi-equivalent} if there exist DG-categories $\kc_1,\ldots, \kc_n$ and a chain of quasi-equivalences $\begin{xy}\xymatrix{\ka & \kc_1 \ar[l] \ar[r] & \cdots & \kc_n \ar[l] \ar[r] & \kb.}\end{xy}$

A DG-functor $\Phi\colon \ka \rightarrow \kb$ is a \emph{DG-equivalence} if it is fully faithful and for every object $B \in \kb$ there is a closed isomorphism of degree 0 between $B$ and an object of $\Phi(\ka)$.
\end{definition}

We also have to recall the following construction from \cite{BK}.

\begin{definition}
Let $\ka$ be a DG-category. Define the \emph{pretriangulated hull} $\ka^{pretr}$ of $\ka$ to be the following category. Its objects are formal expressions $(\oplus_{i=1}^n C_i[r_i],q)$, where $C_i \in \ka$, $r_i \in \IZ$, $n\geq 0$, $q=(q_{ij})$, $q_{ij} \in \Hom(C_j, C_i)[r_i-r_j]$ is homogeneous of degree 1, $q_{ij}=0$ for $i \geq j$, $dq+q^2=0$. If $C=(\oplus_{j=1}^n C_j[r_j],q)$ and $C'=(\oplus_{i=1}^m C'_i[r'_i],q')$ are objects in $\ka^{pretr}$, then the $\IZ$-graded $K$-module $\Hom(C,C')$ is the space of matrices $f=(f_{ij})$, $f_{ij} \in \Hom(C_j,C'_i)[r'_i-r_j]$ and the composition map is matrix multiplication. The differential $d\colon \Hom(C,C') \rightarrow \Hom(C,C')$ is defined by $d(f)=(df_{ij})+q'f-(-1)^lfq$ if $\deg f_{ij}=l$. The category $\ka$ is called \emph{pretriangulated} if the natural fully faithful functor $\Psi\colon \ka \rightarrow \ka^{pretr}$ is a quasi-equivalence and $\ka$ is \emph{strongly pretriangulated} if $\Psi$ is a DG-equivalence. 
\end{definition}

The reason for introducing the pretriangulated hull is that its homotopy category is always triangulated. Thus, we have the following

\begin{definition}
Let $\ka$ be a DG-category. The associated triangulated category is $\ka^{tr}:=H^0(\ka^{pretr})$.
\end{definition}

%\begin{remark}
%There is an equivalent way of describing the pretriangulated hull. To do this, recall that for any $K$-linear DG-category $\ka$ the category of contravariant DG-functors $\Hom(\ka^{op}, C_{DG}(K))$ is called the category of right \emph{DG-modules} and denoted by $\Mo(\ka)$. As in the classical setting there is a Yoneda embedding $\ka \rightarrow \Mo(\ka)$ and an element in the image is called \emph{representable}. A DG-module $\Phi$ is called \emph{semi-free} if there exists a filtration $0=\Phi_0 \subset \Phi_1 \subset \cdots \subset \Phi$ such that $\Phi_{k+1}/ \Phi_{k}$ is isomorphic to a direct sum of shifts of representable modules. A semi-free DG-module is \emph{finitely generated} if $\Phi_{n}=\Phi_{n+1}$ for all $n \gg 0$ and $\Phi_{k+1}/ \Phi_{k}$ is a finite direct sum. As explained in \cite{BK} there is a canonical embedding $\ka^{pretr} \rightarrow \Mo(\ka)$ and under this embedding $\ka^{pretr}$ is DG-equivalent to the category of semi-free finitely generated DG-modules (cf.\ \cite{Drinfeld}).  
%\end{remark}

%It is possible to localise the 2-category of DG-categories with respect to the quasi-equivalences and the morphisms between two DG-categories $\ka$ and $\kb$ in the localisation are in a natural bijection with isomorphism classes of so-called quasi-functors, see \cite{Toen}. A \emph{quasi-functor} can be represented by a DG-functor $\ka \rightarrow \Mo(\kb)$ whose essential image consists of DG-functors which are quasi-isomorphic to representable ones. Note that any quasi-functor induces a functor $H^0(\ka)\rightarrow H^0(\kb)$.
Finally we have the following notion.

\begin{definition}
Let $\kt$ be a triangulated category. An \emph{enhancement} of $\kt$ is a pair $(\ka,\epsilon)$, where $\ka$ is a pretriangulated DG-category and $\begin{xy}\xymatrix{\epsilon\colon H^0(\ka) \ar[r]^(.6)\sim & \kt}\end{xy}$ is an equivalence of triangulated categories.

The category $\kt$ is said to have a unique enhancement if it has one and for two enhancements $(\ka,\epsilon)$ and $(\ka',\epsilon')$ there exists a quasi-functor (see \cite{LO}) $\phi\colon \ka \rightarrow \ka'$ which induces an equivalence $H^0(\phi)\colon H^0(\ka)\rightarrow H^0(\ka')$. One then calls the two enhancements \emph{equivalent}. Two enhancements are called \emph{strongly equivalent} if there exists a quasi-functor $\phi$ such that $\epsilon'\circ H^0(\phi)$ and $\epsilon$ are isomorphic. 
\end{definition}

If $\kt\cong \Db(X)$ for $X$ a smooth projective variety (or a smooth stack), the enhancement one usually works with is
\begin{equation}
\ka:=\Db_{DG}(X):=C_{DG}(I(X)),
\end{equation}
where $C_{DG}(I(X))$ is the DG-category of bounded-below complexes of injective sheaves with bounded coherent cohomology.

%A reformulation of the above is the following: Two enhancements are identified if there exists a chain as in Definition \ref{quasi} where all the $\kc_i$ are enhancements as well.

%\begin{remark}
%According to (the more general) \cite[Thm.\ 9.9]{LO} the category $\Db(X)$ has a strongly unique enhancement if $X$ is a smooth and projective variety.
%\end{remark}

Denote the two projections from $X\times X$ to $X$ by $q$ and $p$. Let $F\colon \Db(X)\rightarrow \Db(X)$ be an equivalence. By results of Orlov (\cite{Orlov}, \cite{Orlov2}) we know that $F$ is of \emph{Fourier--Mukai type}, that is, there exists a unique (up to isomorphism) object $\kp \in \Db(X\times X)$, called the \emph{kernel}, such that $F\cong \Phi_\kp$, where $\Phi_\kp(E)=p_*(q^*E\otimes \kp)$ for any $E\in \Db(X)$. It is clear that any equivalence of FM-type lifts to a DG-endofunctor of $\Db_{DG}(X)$. In fact, the DG-lifts of the three standard autoequivalences (shifts, automorphisms and line bundle twists) do lift to DG-equivalences of the above described enhancement, but for a general autoequivalence this is not clear. Given an arbitrary triangulated category the existence and/or uniqueness of an enhancement is not known (but see \cite{LO} for several results) and the question whether an exact functor lifts to the DG-enhancement is also open.

\section{Linearisations}

Let $G$ be a group and $\kc$ be any category. The following notions are based on Deligne's article \cite{Deligne}. A \emph{weak action} of $G$ on $\kc$ is the assignment of an autoequivalence $g^*$ to any element $g\in G$ such that there exist isomorphisms of functors $c_{g,h}\colon (gh)^*\cong h^*g^*$ for all $g,h \in G$. Note that this, in particular, implies that $1^*\cong \id_\kc$. An \emph{action} of $G$ on $\kc$ is a weak action such that the isomorphisms $c_{g,h}$ satisfy an associativity condition:
\[
\begin{xy}
\xymatrix{ (ghi)^* \ar[d] \ar[r] & i^*(gh)^*\ar[d] \\
					(hi)^*g^* \ar[r] & i^*h^*g^*.}
\end{xy}
\]

Of course, there is a ``covariant'' version of the above definition.

%Working with DG-categories we will adopt a notion which looks slightly different at first sight.

%The reason for this notion is that the lift of an equivalence of $\kt$ can rarely be expected to be a DG-equivalence in general. Note that one usually works with the localization of the category of DG-categories with respect to quasi-equivalences and hence our restriction to honest DG-functors might be considered a bit unfortunate but is necessary for our purposes.

\begin{definition}
A \emph{linearisation} of an object $A$ of a category $\ka$ consists of a collection of morphisms $\lambda_g\colon A\rightarrow g^*(A)$ in $\ka$ for each $g\in G$ which satisfy the following: $\lambda_1=\id$ and $\lambda_{gh}=h^*(\lambda_g)\lambda_h$, that is 
\[\begin{xy}\xymatrix{ A \ar@/_1.5pc/[rrrr]^{\lambda_{gh}}\ar[rr]^{\lambda_h} && h^*(A) \ar[rr]^{h^*(\lambda_g)} && h^*(g^*(A)).}\end{xy}\]

A pair $(A,\lambda_g)$ will be called a \emph{linearised object}. We define a morphism between two linearised objects $(A,\lambda_g)$ and $(A',\lambda'_g)$ to be a $G$-invariant morphism, that is, a morphism $\varphi\colon A\rightarrow A'$ in $\ka$ such that the following diagram commutes in $\ka$ for all $g \in G$:
\[
\begin{xy}
\xymatrix{ A \ar[d]_{\lambda_g} \ar[r]^{\varphi} & A' \ar[d]^{\lambda'_g}\\
					g^*(A) \ar[r]^{g^*(\varphi)} & g^*(A').}
\end{xy}
\]
\end{definition}

Thus, we have a category of linearised objects $\ka^G$. The linearised category inherits properties of $\ka$ if $\ka$ is ``rigid'' enough, for example we have the following

\begin{prop}\label{lin-abelian}
If $G$ acts on an abelian category $\ka$ by exact autoequivalences, then $\ka^G$ is also an abelian category.
\end{prop}

\begin{proof}
The existence of direct sums and the zero object is obvious. Given a $G$-invariant morphism $\varphi\colon (A,\lambda_g)\rightarrow (A',\lambda'_g)$ the universal properties of the kernel and the cokernel in $\ka$ ensure that $\ker(\varphi)$ and $\coker(\varphi)$ are canonically linearised and the respective morphisms are $G$-invariant. Hence, kernels and cokernels exist. The most interesting part is the equality of the image and the coimage: By the above arguments these objects are linearised and it can be checked, using that the kernel is a monomorphism, that the isomorphism in $\ka$ between $\im(\varphi)$ and $\coim(\varphi)$ is $G$-invariant. 
\end{proof}

If $G$ acts on a triangulated category $\kt$ by exact autoequivalences, then, given a $G$-invariant map $\varphi$, it is not clear how to linearise a cone of this map. Therefore we will take the detour via DG-categories.

There is the following easy result.

\begin{prop}
If $\ka$ is a DG-category with an action by a group $G$, then the category of linearised objects as defined above is a DG-category.
\end{prop}

\begin{proof}
We only need to prove that the space of morphisms has the structure of a complex. This boils down to proving that for any morphism $\varphi\colon (A,\lambda_g)\rightarrow (A',\lambda'_g)$ the morphism $d(\varphi)\colon A \rightarrow A'$ is also compatible with the linearisations. Since $\lambda_g$ is a DG-isomorphism for all $g \in G$ we, in particular, have that any $\lambda_g$ is closed, that is, $d(\lambda_g)=0$ for all $g\in G$. Now one only has to use the Leibniz rule, the fact that $\lambda'_g$ has degree $0$ and that any $g \in G$ defines a DG-functor and therefore is compatible with the differentials:
\[\lambda'_g \circ \varphi=g^*(\varphi)\circ \lambda_g \Longrightarrow \lambda'_g\circ d(\varphi)=d(\lambda'_g \circ \varphi)=d(g^*(\varphi)\circ \lambda_g)=g^*(d(\varphi))\circ\lambda_g.\]
\end{proof}

\begin{remark}
Given an action of an algebraic group $G$ on a variety $X$ denote the action by $\sigma\colon G\times X\rightarrow X$ and the multiplication by $\mu\colon G\times G\rightarrow G$. One defines a linearisation of a sheaf $\kf$ in this case to be an isomorphism $\lambda\colon \sigma^*\kf\rightarrow p_2^*\kf$ of $\ko_{G\times X}$-modules subject to the cocycle condition $(\mu\times \id_X)^*\lambda=p_{23}^*\lambda\circ (\sigma\times \id_G)^*\lambda$, where $p_2\colon G\times X \rightarrow X$ and $p_{23}\colon G\times G\times X\rightarrow G\times X$ are the projections. For a finite (or discrete) group this reduces to isomorphisms $\kf\rightarrow g^*\kf$ as above. Hence, we should not expect the above construction to be compatible with geometry in the case of an arbitrary group (there is a notion of an equivariant derived category in this case, see \cite{BL}).
\end{remark}

\begin{definition}
We define the \emph{forgetful functor} $\Forg$ as the functor $\ka^G\rightarrow \ka$ which forgets the linearisations. The \emph{inflation functor} $\Inf$ from $\ka$ to $\ka^G$ is the functor which on objects is defined by $A\mapsto \oplus_g g^*(A)$.
\end{definition}

\begin{remark}
Given a subgroup $H\subset G$, we have an obvious DG-functor $\ka^G\rightarrow \ka^H$. This functor is clearly faithful, but the case $H=\left\{1\right\}$ shows that it is not essentially surjective in general.
\end{remark}

Note that, since there always exists an extension of a DG functor to a DG functor on the pretriangulated hulls, the action of $G$ on $\ka$ naturally extends to an action on $\ka^{pretr}$. Explicitly, the action of $G$ on the formal shifts is clear and for an object $(\oplus_{i=1}^n C_i[r_i],q)$ in $\ka^{pretr}$ we simply apply an element $g$ in $G$ to the components of $q$. One then checks that the condition $d(q)+q^2=0$ induces that $d(g^*(q))+g^*(q)^2=0$ (to see this note that $g^*(\alpha\beta)=g^*(\alpha)g^*(\beta)$ and similarly for sums; hence $g^*$ is compatible with matrix multiplication) and hence the action is well-defined. The definition of the action on morphisms is similarly straightforward. %Since $G$ acts by quasi-equivalences on $\ka$, it is still the case for its action on $\ka^{pretr}$. We have the following useful

\begin{prop}\label{linearisations-pretriangulated}
If $G$ is a group acting on a strongly pretriangulated DG-category $\ka$, then $\ka^G$ is also strongly pretriangulated.
\end{prop}

\begin{proof}
Consider the formal shift of a linearised object $(A,\lambda_g)[1]$. We know that $A[1]$ is DG-isomorphic to an object $A'$ of $\ka$ and $\lambda_g$ induces a linearisation $\lambda'_g$ of $A'$. It is then clear that $(A,\lambda_g)[1]$ is DG-isomorphic to $(A',\lambda'_g)$. The reasoning for the cone of a closed degree zero morphism is similar.
\end{proof}

\begin{remark}
We cannot show that the statement of the lemma holds true without the ``strongly'' assumption.

One would need to show that for every $k \in \IZ$ and any $(A,\lambda_g)\in \ka^G$ the object $(A,\lambda_g)[1]$ is homotopy equivalent to an object in $\ka^G$ and similarly that the cone of any closed degree zero morphism also has this property. Let $\varphi\colon (A,\lambda_g)\rightarrow (A',\lambda'_g)$ be a closed degree zero morphism in $\ka^G$. For simplicity write $\widetilde{A}$ for $(A,\lambda_g)$ and similarly $\widetilde{A'}$ for $(A',\lambda'_g)$. By definition the cone $C(\varphi)$ is the object $(\widetilde{A'}[1]\oplus \widetilde{A},q)$, where $q$ is the $(2\times 2)$-matrix with $q_{12}=\varphi$ and 0 otherwise. On the other hand the cone of $\varphi$ considered in $\ka^{pretr}$, which we will denote by $C'(\varphi)$, has a linearisation $\gamma_g$ given by
\[
\begin{pmatrix}
\lambda'_g & 0 \\ 0 & \lambda_g 
\end{pmatrix}
\]   
Using that $\ka$ is pretriangulated we know that $C'(\varphi)$ is homotopy equivalent to an object $D$ of $\ka$, but the linearisation $\gamma_g$ does not necessarily induce a linearisation $\delta_g$ on $D$.
\end{remark}

We can now give our definition of the linearised triangulated category.

%\begin{definition}
%Let $\kt$ be a triangulated category possessing an enhancement $\ka$ and let $G\in \Aut(\kt)$ be a finite group of exact autoequivalences of $\kt$ which lift to quasi-equivalences of $\ka$. The \emph{quotient} of $\kt$ by $G$, denote by $\kt_\ka/G$, is defined to be $H^0((\ka^G)^{pre-tr})$.
%\end{definition}
\begin{definition}
Let $\kt$ be the homotopy category of a pretriangulated DG-category $\ka$ and let $G$ be a group acting on $\ka$ and hence on $\kt$. The \emph{linearisation} of $\kt$ by $G$, denoted by $\kt_\ka^G$, is defined to be $H^0((\ka^G)^{pretr})$.
\end{definition}

Since the enhancement will usually be clear from the context, we will simply write $\kt^G$ instead of $\kt^G_\ka$.

\begin{remark}
Given an additive category $\kt$ and an automorphism $\Phi$ there is an ``orbit category'' $\kt/\Phi$, which has the same objects as $\kt$ and the morphisms between two objects $A$ and $B$ are given by
\[\bigoplus_{n \in \IZ} \Hom_\kt(A,\Phi^nB).\]
If $\kt$ is a triangulated category, then $\kt/\Phi$ is not triangulated in general, but under some fairly strong assumptions it does have a triangulated structure, see \cite{Keller3}. In particular, the mentioned assumptions are not satisfied by $\Db(X)$ for $X$ a smooth projective variety. In general, the orbit category does not seem to capture quotients: Consider a smooth projective variety $X$ with an action by an automorphism such that $X/G$ is smooth and the category $\Coh(X)$, considered for simplicity as an additive category. Then the orbit category is $\pi_*\Coh(X)$ and not $\Coh(X/G)$. 
\end{remark}

Under certain assumptions the construction is functorial. Namely, assume that $G$ acts on two pretriangulated DG-categories $\ka$ and $\ka'$ and that we are given an equivariant DG-functor $F\colon \ka \rightarrow \ka'$ meaning that the following diagram commutes for all $g\in G$:
\[\begin{xy}\xymatrix{ \ka \ar[d]_{g^*} \ar[r]^\Phi & \ka' \ar[d]^{g^*} \\
												\ka \ar[r]^\Phi & \ka'.}\end{xy}\]

Then $\Phi$ induces a functor $\Phi^G\colon\ka^G \rightarrow (\ka')^G$ by sending $(A,\lambda_g)$ to $(\Phi(A),\Phi(\lambda_g))$ and similarly for maps. Using the induced functor on the pretriangulated hulls, we get an exact functor $H^0((\Phi^G)^{pretr})\colon \kt^G\rightarrow (\kt')^G$. 

As a particular example of this consider the situation of a triangulated subcategory. So, suppose we are given a triangulated category $\kt=H^0(\ka)$ with a group action and a triangulated subcategory $\kt'$ such that the action of $G$ restricts to an action on the natural enhancement of $\kt'$ (i.e.\ the objects of the enhancement are those objects $A'$ in $\ka$ such that $H^0(A')\in \kt'$). Then $(\kt')^G$ is a triangulated subcategory of $\kt^G$.

It is, of course, an interesting question whether the construction depends on the choice of the enhancement, that is, if $\kt$ can be written as the homotopy category of two distinct DG-categories $\ka$ and $\kb$, is then $\kt^G_\ka$ equivalent to $\kt^G_\kb$. We will now see that this can be indeed shown, albeit under some assumptions.

\begin{lem}
Let $\ka$ and $\kb$ be two pretriangulated DG-categories and let $\Phi\colon \ka\rightarrow \kb$ be an equivariant quasi-equivalence. Assume that for all $(A,\lambda_g), (A',\lambda'_g) \in \ka^G$ and for all $\varphi \in \Hom_\ka(A,A')$ we have that $g^*(d\varphi)\lambda_g=\lambda'_gd\varphi$ implies that $g^*(\varphi)\lambda_g=\lambda'_g\varphi$ (we call this condition $(\ast)$) and similarly for $\kb$. 
%Furthermore, assume that the categories $H^\bullet(\ka)$ and $H^\bullet(\kb)$ are of finite type, that is, the morphism spaces are finite-dimensional. 
Then $\Phi^G$ is quasi fully faithful. 
\end{lem}

\begin{proof}
Condition $(\ast)$ ensures that $H^i_{\ka^G}((A,\lambda_g),(A',\lambda'_g))$ embeds into $H^i_\ka(A,A')$ for all $i$. The commutativity of the following diagram (for all $i \in \IZ$) 
\[\begin{xy}\xymatrix{H^i_{\ka^G}((A,\lambda_g),(A',\lambda'_g)) \ar[d] \ar[r]^(.4){H^i(\Phi^G)} & H^i_{\kb^G}((\Phi(A),\Phi(\lambda_g)),(\Phi(A'),\Phi(\lambda'_g))) \ar[d]\\
H^i_\ka(A,A') \ar[r]^{H^i(\Phi)}& H^i_\kb(\Phi(A),\Phi(A'))}\end{xy}
\]
combined with the injectivity of the left vertical map and the fact that the lower map is an isomorphism implies the injectivity of $H^i(\Phi^G)$.

Next, take an element $\overline{\psi}$ in $H^i_{\kb^G}((\Phi(A),\Phi(\lambda_g)),(\Phi(A'),\Phi(\lambda'_g)))$. Since by $(\ast)$ the right vertical map is injective this gives an element in $H^i_\kb(\Phi(A),\Phi(A'))$ and hence a unique element $\overline{\varphi}$ in $H^i_\ka(A,A')$. We need to check that in fact $\overline{\varphi}$ is in $H^i_{\ka^G}((A,\lambda_g),(A',\lambda'_g))$. By assumption and quasi-faithfulness of $\Phi$ we have that 
\[g^*(\varphi)\lambda_g-\lambda'_g\varphi=d(\alpha_g)\]
for some $\alpha_g$ (here $\varphi$ is a representative of $\overline{\varphi}$). Differentiating the above equation gives that $d(\varphi)$ commutes with the linearisations and therefore, by $(\ast)$, $\varphi$ does. Hence, $\Phi^G$ is quasi fully faithful. 
\end{proof}

\begin{prop}\label{ind-enh}
In addition to the assumptions of the previous lemma let furthermore $\Psi\colon \kb \rightarrow \ka$ be an adjoint equivariant quasi-equivalence. Then $\kt^G_\ka$ and $\kt^G_\kb$ are equivalent.
\end{prop}

\begin{proof}
Since $\Phi$ and $\Psi$ are adjoint, so are $\Phi^G$ and $\Psi^G$. Both these functors are quasi fully faithful and hence define equivalences on the homotopy categories. Now use that if a DG-functor $F$ is a quasi-equivalence, then so is $F^{pretr}$.
\end{proof}

\begin{remark}
One probably should not expect that our construction is independent of the choice of the enhancement in general, since linearisations do use the DG-structure. Of course, it may then be asked what the relation between $\kt^G_\ka$ and $\kt^G_\kb$ is.
\end{remark}
%\begin{proof}
%We know that $\ka^G$ and $\kb^G$ are quasi-equivalent. But then so are $(\ka^G)^{pretr}$ and $(\kb^G)^{pretr}$, see \cite[Prop.\ 2.5]{Drinfeld}.
%\end{proof}

It is difficult to produce new interesting autoequivalences of finite order in the geometric setting. The case of a finite group of automorphisms will be settled in the next section. The next case is the action of the group generated by a line bundle of finite order, which will be partially dealt with in Section \ref{Applications}. Basically, these are the only examples we have at our disposal. One might hope to produce new ones by conjugating the action of one of the above mentioned groups but we will now see that this does not give anything new. 

Let $G$ act on a pretriangulated DG-category $\ka$, let $\Phi$ be a DG-equivalence of $\ka$, consider the action of $G$ given by $g\mapsto \Phi^{-1}\circ g^*\circ \Phi$ and denote the DG-category linearised with respect to this action by $\ka^{\widetilde{G}}$. Then we have the
\begin{prop}
The categories $\ka^{G}$ and $\ka^{\widetilde{G}}$ are DG-equivalent.
\end{prop}

\begin{proof}
The functor $F$ sending an object $(A,\lambda_g)$ in $\ka^G$ to $(\Phi^{-1}(A),\Phi^{-1}(\lambda_g))$ is easily seen to be a DG-equivalence. 
\end{proof}
%Thus, we would like to consider conjugation by a quasi-equivalence. Of course, this is not a priori possible and we have to introduce several assumptions. Let $\Phi$ be a quasi-equivalence and assume that it has an adjoint $\Psi$ (this implies that $H^0(\Psi)\cong H^0(\Phi)^{-1}$). This is satisfied in the geometric situation if $\Phi$ is a lift by \cite[Lem.\ 7.5]{BLL}. Under these assumptions we can now state the

%\begin{prop}
%Let an action of the group $G$ on a triangulated category $\kt$ and its extension to an enhancement $\ka$ be given and consider the action conjugated by a quasi-equivalence $\Phi$ which has an adjoint $\Psi$. Denoting the category linearised with respect to the latter action by $\ka^{\widetilde{G}}$, we have $H^0(\ka^G)\cong H^0(\ka^{\widetilde{G}})$.
%\end{prop} 

%\begin{proof}
%We will first define a DG-functor from $\ka^G$ to $\ka^{\widetilde{G}}$. Applying $\Phi$ to $(A,\lambda_g) \in \ka^G$, we get a pair $(\Phi(A),\Phi(\lambda_g))$. Now, a linearisation with respect to $\widetilde{G}$ is a morphism $\beta_g\colon A'\rightarrow \widetilde{g}^*(A')=\Phi g\Psi (A')$ for any $g \in G$ satisfying the linearisation relation in the homotopy category. Using that $\Psi$ and $\Phi$ are adjoint we see that $\Phi(\lambda_g)$ does exactly that for $A'=\Phi(A)$. 

%The equivalence of the homotopy categories follows along the lines of the proof of Lemma \ref{indep-enhancement}. 
%\end{proof} 

\section{The case of automorphisms}\label{classical}

Of course, one has to check that the above procedure produces the derived category of $[X/G]$ if $G$ is a finite group of automorphisms of $X$. Denote the cardinality of $G$ by $n$. There are canonical isomorphisms $(gh)^*\cong h^*g^*$ and these will be used to define the action of $G$.

We first recall some useful facts. There is an equivalence $\(Q)Coh([X/G])\cong \(Q)Coh^G(X)$, where the latter is the category of linearised sheaves. The quotient morphism $\pi\colon X \rightarrow [X/G]$ is flat, hence $\pi^*$ is exact. 
%The exact pushforward functor $\pi_*$ corresponds to the \emph{inflation functor} $\kf \rightarrow \oplus_g g^*(\kf)$, where one uses the canonical linearisation on the latter object. The \emph{forgetful functor} sending $(\kf,\lambda_g)$ to $\kf$ corresponds to $\pi^*$. 
Using the adjunction
\[\Hom(\pi^*(-),-)\cong \Hom(-,\pi_*(-))\]
and the exactness of $\pi^*$ we see that the pushforward of an injective sheaf on $X$ is an injective sheaf on $[X/G]$.

Now consider $\kt=\Db(X)$ as the homotopy category of $\Db_{DG}(X)$. Then an object of $\ka^G$ is by definition a complex as above together with chain isomorphisms $\lambda_g$ satisfying the linearisation relation. Since $\ka$ is strongly pretriangulated, so is $\ka^G$ (see Proposition \ref{linearisations-pretriangulated}).

We need the following 

\begin{lem}
Let $\kf'=(\kf,\lambda_g)$ be an injective sheaf on $[X/G]$. Then $\kf=\pi^*(\kf')$ is an injective sheaf on $X$.
\end{lem}

\begin{proof}
Denote $[X/G]$ by $Y$ and the Serre functor by $S_X$ resp.\ $S_Y$. The result follows from the following computation
\begin{align*}
\Hom_X(-,\pi^*(\kf'))&\cong\Hom_X(\pi^*(\kf'),S_X(-))^\vee \cong \Hom_Y(\kf',\pi_*S_X(-))^\vee\cong\\
&\cong\Hom_Y(\pi_*S_X(-),S_Y(\kf'))\cong \Hom_Y(S^{-1}_Y\pi_*S_X(-),\kf').
\end{align*}
Note that the functor $S_Y^{-1}\pi_* S_X$ takes sheaves to sheaves, since the shifts cancel out. Furthermore, $\kf'$ is injective by assumption, $\pi_*$ is exact and tensoring with line bundles is also exact, hence we conclude that $\Hom(-,\pi^*(\kf'))$ is an exact functor and therefore $\kf=\pi^*(\kf,\lambda_g)$ is injective.
\end{proof}

\begin{remark}
In our situation one has that $\pi^*\omega_Y\cong \omega_X$ by the ramification formula $\omega_X\cong \pi^*\omega_Y\otimes \ko(R)$. Hence the projection formula gives that 
\[S_Y^{-1}\pi_* S_X(\kf)\cong \pi_*(\kf\otimes \omega_X)\otimes \omega_Y^{-1}\cong \pi_*(\kf).\]
The forgetful functor $\Forg$ corresponds to $\pi^*$ and the inflation functor $\Inf$ corresponds to $\pi_*$. Hence, $\Forg\circ\Inf\cong \oplus_{g\in G} g^*\cong \pi^*\pi_*$.

Using this, one can give another proof of the above lemma:
Consider the following diagram in $\QCoh(X)$:
\[\begin{xy}\xymatrix{ & & \pi^*(\kf) \\
												0 \ar[r] & \kp \ar[r]^\psi \ar[ur]^\phi & \kq.}\end{xy}\]
We need to construct a map $\theta\colon \kq \rightarrow \pi^*(\kf)$ making the diagram commutative. Apply $\pi_*$ to the diagram. Since $\pi_*\pi^*(\kf)=\kf\otimes \pi_*\ko_X$ and $\pi_*\ko_X=\kl$ is locally free, $\pi_*\pi^*(\kf)$ is an injective sheaf again and therefore we get a diagram
\[\begin{xy}\xymatrix{& & \kf\otimes \kl \\
												0 \ar[r] & \pi_*\kp \ar[r]^{\pi_*(\psi)} \ar[ur]^{\pi_*(\phi)} & \pi_*\kq \ar[u]_\alpha.}\end{xy}\]
Applying $\pi^*$ and using that $\pi^*\pi_*=\oplus_g g^*$ we get
\[\begin{xy}\xymatrix{ & & (\pi^*(\kf))^{\oplus n} \\
												0 \ar[r] & \oplus_g g^*(\kp) \ar[r]^{\oplus_g g^*(\psi)} \ar[ur]^{\oplus_g g^*(\phi)} & \oplus_g g^*(\kq) \ar[u]_{\pi^*(\alpha)}.}\end{xy}\]
Denoting the inclusion of $\kq$ in $\oplus_g g^*(\kq)$ by $\iota$ and the first projection from $(\pi^*(\kf))^{\oplus n}$ to $\pi^*(\kf)$ by $p_1$, the wanted morphism $\theta$ is then $p_1 \circ \pi^*(\alpha)\circ \iota$.
\end{remark}

Let $\Db_{DG}([X/G])$ be the enhancement of $\Db([X/G])$ by injective sheaves. Using the above lemma we can construct a functor from $\Db_{DG}([X/G])$ to $\ka^G$: If a complex of injective linearised sheaves $(\kf^i,\lambda^i_g)$ is given, then we can send it to the complex having the $\kf^i$ as terms and the linearisation of this complex is given termwise by the $\lambda^i_g$. A map in $\kb$ is simply sent to itself. Clearly, this is fully faithful, so in particular we get: 
%Consider the following diagram in $\QCoh(X)$:
%\[\begin{xy}\xymatrix{ & & \pi^*(\kf) \\
												%0 \ar[r] & \kp \ar[r]^\psi \ar[ur]^\phi & \kq.}\end{xy}\]
%We need to construct a map $\theta\colon \kq \rightarrow \kf$ making the diagram commutative. Apply $\pi_*$ to the diagram. Since $\pi_*\pi^*(\kf)=\kf\otimes \pi_*\ko_X$ and $\pi_*\ko_X=\kl$ is locally free, $\pi_*\pi^*(\kf)$ is an injective sheaf again and therefore we get a diagram
%\[\begin{xy}\xymatrix{& & \kf\otimes \kl \\
											%	0 \ar[r] & \pi_*\kp \ar[r]^{\pi_*(\psi)} \ar[ur]^{\pi_*(\phi)} & \pi_*\kq \ar[u]_\alpha.}\end{xy}\]
%Applying $\pi^*$ and using that $\pi^*\pi_*=\oplus_g g^*$ we get
%\[\begin{xy}\xymatrix{ & & (\pi^*(\kf))^{\oplus n} \\
												%0 \ar[r] & \oplus_g g^*(\kp) \ar[r]^{\oplus_g g^*(\psi)} \ar[ur]^{\oplus_g g^*(\phi)} & \oplus_g g^*(\kq) %\ar[u]_{\pi^*(\alpha)}.}\end{xy}\]
%Denoting the inclusion of $\kq$ in $\oplus_g g^*(\kq)$ by $\iota$ and the first projection from $(\pi^*(\kf))^{\oplus n}$ to $\pi^*(\kf)$ by $p_1$, the wanted morphism $\theta$ is then $p_1 \circ \pi^*(\alpha)\circ \iota$. 

\begin{prop}\label{classical-picture}
There exists a DG-equivalence $\Phi\colon \Db_{DG}([X/G]) \rightarrow \ka^G$. Hence, $\Db(X)^G\cong \Db[X/G]$.\qqed
\end{prop}

The previous discussion can be used to deduce the following

\begin{cor}\label{FM-quotient}
Let $Y$ be a Fourier--Mukai partner of $X$, let $G$ be a finite group of automorphisms acting on $Y$ and fix an equivalence $\begin{xy}\xymatrix{F\colon \Db(X) \ar[r]^(.5)\sim &  \Db(Y).}\end{xy}$ If we let the group $G$ acts on $\Db(X)$ by $F^{-1}\circ g^* \circ F$ for any $g \in G$, then $\Db(X)^G$ is equivalent to $\Db([Y/G])$.
\end{cor}

\begin{proof}
If we write $\Db(X)$ as the homotopy category of $\Db_{DG}(Y)$, the result follows immediately.
\end{proof}

Let us now consider the following situation. Let $G$ act on $X$ without fixed points and consider the quotient map $\pi\colon X\rightarrow X/G=:Y$. Take the structure sheaf $\ko_Z$ of a subvariety $Z\subset Y$ and consider the full triangulated subcategory $\kt$ generated by $\ko_Z$ in $\Db(Y)$. Pulling $\ko_Z$ up to $X$ gives the direct sum of $\ko_{Z_i}$, where $\cup Z_i$ is the preimage of $Z$ (for example, $Z$ could be a rational curve on an Enriques surface and then we get two rational curves $Z_1$ and $Z_2$ on the covering K3 surface). Consider the triangulated subcategory of $\Db(X)$ generated by the sheaves $\ko_{Z_i}$ and denote it by $\kt'$. Then we have 

\begin{prop}
There exists an equivalence $(\kt')^G\cong \kt$. 
\end{prop}

\begin{proof}
Clearly, $\oplus_i \ko_{Z_i}$ can be linearised, hence $\ko_Z \in (\kt')^G$ and therefore $\kt \subset (\kt')^G$. On the other hand, $\kt'$ is built from the $\ko_{Z_i}$ by taking iterated extensions and shifts, whose support is always contained in the union of the $\ko_{Z_i}$. Therefore the only objects which can be linearised are generated by $\oplus_i \ko_{Z_i}$ and therefore the other inclusion holds as well. 
\end{proof}

\section{Applications}\label{Applications}
%\begin{prop}
%Let $\kt$ be a triangulated category equipped with an action of a finite group $G$ such that the action lifts to an enhancement $\ka$. Furthermore, let $\Phi$ be a quasi-equivalence of $\ka$ such that $(H^0(\Phi))^{-1}$ can be represented by a DG-functor $\Psi$, by abuse of notation denoted by $\Phi^{-1}$. Consider the action of $G$ given by $g\mapsto \Phi\circ g^*\circ \Phi^{-1}$ and denote the DG-category linearised with respect to this action by $\ka^{\widetilde{G}}$. Then $H^0(\ka^G)\cong H^0(\ka^{\widetilde{G}})$.
%\end{prop}

\subsection{Linearisations with respect to a torsion canonical bundle}
Having checked the above we will look at the case of a linearisation with respect to a line bundle twist. Recall (Proposition \ref{lin-abelian}) that given an action of a finite group $G$ on an abelian category, the category of linearised objects is also abelian. We will now consider the following situation. Let $S$ be a variety whose canonical bundle is of finite order, fix an isomorphism $f\colon\omega_S^n\cong \ko_S$ and consider the global spectrum $\widetilde{S}$ of the corresponding sheaf of $\ko_S$-algebras. Using $f$, the sheaf $\ko_S\oplus \omega_S\oplus\cdots \omega_S^{n-1}$ becomes an $\ko_S$-algebra and $\widetilde{S}$ has a fixed point-free automorphism $\tau$ of order $n$ corresponding to the action of $\omega_S$. Denote the quotient morphism $\widetilde{S}\rightarrow S$ by $\pi$. The isomorphism $f$ induces an action of $\IZ/n\IZ$ on the category of (quasi-)coherent sheaves on $S$ by sending $1$ to the functor $(-)\otimes \omega_S$. We can consider sheaves linearised with respect to this action. Any such sheaf $\kf$ has, in particular, the property that $\kf\cong \kf\otimes \omega_S$. This holds, for example, for $\kf=\oplus_{k=0}^{n-1}\omega_S^k$.

Now recall that by \cite[Ex.\ II.5.17]{Hartshorne} there exists an equivalence between $\(Q)Coh(\widetilde{S})$ and the category of (quasi-)coherent $\pi_*\ko_{\widetilde{S}}$-modules. Note that the pullback of $f$ is the identity morphism of $\ko_{\widetilde{S}}$ and that the canonical isomorphism $\pi_*\ko_{\widetilde{S}}\cong\oplus_{k=0}^{n-1}\omega_S^k$ is an isomorphism of $\ko_S$-algebras.

\begin{lem}
Let $\kf$ be a (quasi-)coherent sheaf on $\widetilde{S}$. Then $\pi_*(\kf)$ is linearised with respect to the above described action of $\IZ/n\IZ$ on the category of (quasi-)coherent sheaves on $S$.
\end{lem}

\begin{proof}
We use the projection formula to get an isomorphism
\[\alpha\colon \pi_*(\kf) \cong\pi_*(\kf\otimes \ko_{\widetilde{S}})\cong \pi_*(\kf\otimes \pi^*(\omega_S))\cong \pi_*(\kf)\otimes\omega_S.\]
Since the isomorphism in this formula is canonical, the morphism $(\alpha\otimes \id)\circ \alpha$, resp.\ further compositions are also canonical. On the other hand, composing the $n$-fold composition with $f$ gives the identity because $f$ enters in the very definition of $\widetilde{S}$ and hence of $\pi$. Alternatively, and briefly assuming $n=2$ to simplify notation, one can use that $\pi^*$ is a faithful functor and the fact that the pullback of $f\circ (\alpha\otimes \id)\circ \alpha$ is the identity map of $\pi^*\pi_*\kf\cong \kf\oplus \tau^*(\kf)$.
\end{proof}

Now, given a linearised sheaf $(\kf,\alpha)$ we can define a structure of $\pi_*\ko_{\widetilde{S}}$-module on it by locally setting
\[(s,t)\cdot \gamma:=s\cdot \gamma +\alpha(t\otimes \gamma),\]
for sections of $\ko_S$, $\omega_S$ and $\kf$ respectively.
Clearly, these two constructions are inverse to each other and $\pi_*\ko_{\widetilde{S}}$-linearity translates to invariance with respect to the group action. Hence, we have

%\begin{cor}
%There exists a functor $\pi_*$ from $\(Q)Coh(X)$ to $\(Q)Coh^{\IZ/2\IZ}(S)$.\qqed
%\end{cor}

%We have the following easy

%\begin{prop}
%The functor $\pi_*$ is full and faithful.
%\end{prop}

%\begin{proof}
%Let $\phi\colon \kf\rightarrow \kf'$ be a map such that $\pi_*(\phi)=0$. Since $\pi^*\pi_*=\id\oplus \tau^*$, where $\tau$ is the fixed-point free involution on $X$, we have $0=\pi^*\pi_*(\phi)=\phi\oplus \tau^*(\phi)$. Hence, $\pi_*$ is faithful. To see that it is full, note that
%\[\Hom(\pi_*(\kf),\pi_*(\kf'))\cong\Hom(\pi^*\pi_*(\kf),(\kf'))\cong\Hom(\kf,\kf')\oplus \Hom(\tau^*(\kf),\kf'),\]
%and the latter group does not induce morphisms compatible with the linearisations. 
%\end{proof}

%Let $(\kg,\lambda)$ be an element in $\(Q)Coh^{\IZ/2\IZ}(S)$. Then, in particular, $\kg\cong \kg\otimes \omega_S$ and hence there exists a sheaf $\kf$ on $X$ such that $\pi_*(\kf)=\kg$. %The only thing left to check is that the linearisations are the same, but if they are not, then $\pi_*(\tau^*(\kf))=\kg$ will have this property.
%Thus we have proved the

\begin{prop}
There is an equivalence $\(Q)Coh(\widetilde{S})$ and $\(Q)Coh^{\IZ/n\IZ}(S)$.\qqed 
\end{prop}

\begin{cor}\label{Enriques-K3}
There exists an equivalence $\Db(S)^{\IZ/n\IZ}\cong \Db(\widetilde{S})$.
\end{cor}

\begin{proof}
This follows from the proposition, the fact that $\pi_*(\kf)$ of an injective sheaf $\kf$ is an injective $\pi_*\ko_{\widetilde{S}}$-module and with similar arguments as in Section \ref{classical}.
\end{proof}

\subsection{The category generated by a spherical object}
Let $X$ be a smooth projective variety of dimension $d$. Recall that an object $E$ in $\Db(X)$ is $d$-spherical if $E\otimes \omega_X\cong E$ and if $\Hom(E,E)=\Hom(E,E[d])=\IC$ and $0$ otherwise. More generally, one can define a $d$-spherical object in a triangulated category $\kt$ to be an object $E$ with the property that the graded endomorphism algebra 
\[B=\bigoplus_{ p\in \IZ} \Hom_\kt (E,E[p])\]
is isomorphic to $\IC[s]/(s^2)$ and $s$ is of degree $d$.

For example, any line bundle on a Calabi--Yau variety of dimension $d$ (that is, $\omega_X\cong\ko_X$ and $H^i(X,\ko_X)=0$ for $i\neq 1,d$) is a $d$-spherical object. The interest in these objects stems, in particular, from the fact that one can associate an autoequivalence of $\Db(X)$ to any spherical object, the so-called \emph{spherical twist} (see \cite{ST}). 

Recall that an object $E$ in $\Db(X)$ is \emph{exceptional} if $\Hom(E,E)=\IC$ and $\Hom(E,E[k])=0$ for all $k\neq 0$. Note that spherical objects are usually studied on Calabi--Yau varieties while the most natural environment for exceptional objects are probably Fano varieties. Nevertheless, there is a connection between spherical objects and exceptional objects, see, for example \cite[Subsect.\ 3.3]{ST}. 

In \cite[Thm.\ 2.1]{KYZ} the authors determined the structure of the triangulated category generated by a $d$-spherical object ($d$ is arbitrary). Denoting this category by $\kt_d$, there exists an equivalence between $\kt_d$ and the perfect derived category $\Perf(B)$ of the algebra $B$ introduced above (with $B$ corresponding to the spherical object), which is considered as a DG-algebra with trivial differential. Recall that $\Perf(B)$ is the smallest thick triangulated subcategory of the derived category of $B$ which contains $B$. Since $B$ is a local ring, $\Perf(B)$ is just the homotopy category of bounded complexes of finitely generated free $B$-modules. The enhancement one uses here is the DG-category associated to the additive category of finitely generated free $B$-modules. 

In \cite[Lem.\ 2.3]{FY} the group of autoequivalences of $\kt_d$ which admit DG-lifts was determined. Namely, it is isomorphic to $\IC^*\times \IZ$, where $\IZ$ corresponds to the action of powers of the shift functor and an element $a$ in $\IC^*$ acts by the functor induced by the ring isomorphism $\varphi_a\colon B\rightarrow B$, $s\mapsto a\cdot s$. Note that for any $n\in \IZ$ we have that $G=\IZ/n\IZ$ acts on $\kt_d$ by identifying $G$ with the group of $n$-th roots of unity (clearly, the group action is well-defined).

Let us first spell out what a linearisation in a special case is, namely for $n=2$ and the $B$-module $B$ itself. Note that we are therefore working in the category of modules. Acting by the non-trivial group element means changing the module structure of $B$ as a $B$-module: The element $s$ acts on an element of $B$ not by multiplication with $s$ itself, but with $-s$. Let $\lambda_g\colon B \rightarrow (-1)^*B$ send $1$ to $\alpha+\beta s$. In order for $\lambda_g$ to be $B$-linear, we then must have $\lambda_g(s)=-\alpha s$. Furthermore, if we want $\lambda_g$ to be linearisation, it has to be order 2, so we conclude that $\alpha^2=1$ (and $\beta$ is arbitrary). Thus, for any $\beta \in \IC$ we get linearisations $\lambda_g^{1,\beta}$ ($1 \mapsto 1+\beta s$) and $\lambda_g^{-1,\beta}$ ($1 \mapsto -1+\beta s$) of $B$. However, there is the following

\begin{lem}
The linearised objects $(B,\lambda_g^{1,\beta})$ and $(B,\lambda_g^{1,\beta'})$ are isomorphic for $\beta \neq \beta'$. A similar statement holds for $\lambda_g^{-1,\bullet}$.
\end{lem}  

\begin{proof}
A $B$-linear automorphism of $B$ is a map sending $1$ to $x+ys$ and $s$ to $xs$ for $x\neq 0$. Denote $\beta'-\beta$ by $z$. The map $f$ will commute with the linearisations if we choose $x,y$ such that $z=\frac{2y}{x}$.
\end{proof}

So, we can work with $\beta=0$. Denote $\lambda_g^{\pm 1,0}$ by $\lambda_g^{\pm 1}$.

\begin{lem}
The objects $(B,\lambda_g^1)$ and $(B,\lambda_g^{-1})$ are not isomorphic. The endomorphism ring of $(B,\lambda_g^1)$ resp.\ of $(B,\lambda_g^{-1})$ is isomorphic to $\IC$.
\end{lem}

\begin{proof}
If $f\colon B\rightarrow B$, $1 \mapsto x+ys$ is a map commuting with the linearisations, then we have
\[x+ys=f\lambda_g^1(1)=\lambda_g^{-1}f(1)=\lambda_g^{-1}(x+ys)=-x+ys,\]
hence $x=0$ and $f$ cannot be an isomorphism.

Concerning the second statement we only deal with the first case. Given an $f$ as before we have
\[x+ys=f\lambda_g^1(1)=\lambda_g^1f(1)=x-ys,\]
hence $y=0$ and the endomorphisms of $(B,\lambda_g^1)$ (resp.\ of $(B,\lambda_g^{-1})$) are therefore the morphisms $f$ sending $1$ to $x \in \IC$.
\end{proof}

Combining everything we have

%A similar argument is valid for any free $B$-module of finite rank. Note that this, in particular, implies that only the scalar maps from $B$ to $B$ commute with the linearisation and hence $B$ becomes an exceptional object. This proves the following

\begin{prop}\label{sph-exc}
Let $\kt_d$ be the triangulated category generated by a $d$-spherical object $E$ and consider the action of $G=\IZ/2\IZ$ on it defined above. Then $E$ admits two distinct linearisations $\lambda_g^1$ and $\lambda_g^{-1}$ and the objects $(E,\lambda_g^1)$ and $(E,\lambda_g^1)$ are exceptional in the linearised category. In particular, the linearised category $\kt_d^G$ contains the derived category of $\IC$-vector spaces as a full admissible (the admissibility follows from \cite[Thm.\ 3.2]{Bondal}) triangulated subcategory.\qqed
\end{prop}

Note that the exceptional objects $E_1=(E,\lambda_g^1)$ and $E_2=(E,\lambda_g^{-1})$ are not orthogonal. For larger $n$ the situation becomes more complicated.

\end{document}